\documentclass[10pt, journal]{IEEEtran}

\usepackage{times,amsmath,epsfig,amssymb}
\usepackage{booktabs}
\usepackage[compress]{cite}
\usepackage{tabularx}
\usepackage[font={small}]{caption}
\usepackage{caption}
\usepackage[labelformat=simple,position=b]{subcaption}

\usepackage{amsfonts}
\usepackage{algorithm}
\usepackage{algorithmic}
\usepackage{latexsym}
\usepackage{tikz}
\usetikzlibrary{calc}
\usepackage{bm}
\usepackage{flexisym}
\usepackage{dsfont}
\usetikzlibrary{positioning}
\usetikzlibrary{arrows}
\usetikzlibrary{chains,shapes.multipart}
\usetikzlibrary{shapes}
\usetikzlibrary{automata,}
\usepackage{multirow}
\usepackage{latexsym}
\usepackage[utf8]{inputenc}
\definecolor {processblue}{cmyk}{0,0,0,0.17}
\definecolor{light-gray}{gray}{0.85}
\usepackage{comment}
\usepackage{amsthm}
 \usepackage[usenames,dvipsnames]{pstricks}

\newtheorem{thm}{Theorem}
\newtheorem{lem}[thm]{Lemma}

\newtheorem{remark}{Remark}[thm]

\newcommand{\bzero}{{\boldsymbol 0}}

\begin{document}
\title{A Stochastic Sizing Approach for Sharing-based Energy Storage Applications}

\author{Islam~Safak~Bayram,~\IEEEmembership{Member,~IEEE,}~Mohamed~Abdallah,~\IEEEmembership{Senior~Member,~IEEE,}
~Ali~Tajer,~\IEEEmembership{Senior~Member,~IEEE,}~and
~Khalid~Qaraqe,~\IEEEmembership{Senior~Member,~IEEE}
\vspace{-25 pt}
\thanks{Islam Safak Bayram is with Qatar Environment and Energy Research Institute and College of Science and Technology, Hamad Bin Khalifa University, Education City, Doha, Qatar. Email: ibayram@qf.org.qa.}

 \thanks{Mohamed Abdallah is with the Department of Electrical and Computer Engineering, Texas A\&M University at Qatar, PO Box 23874, Education City, Doha, Qatar and the Department of Electronics and Communications Engineering, Cairo University, Giza, PO Box 12613, Egypt. Email: mohamed.abdallah@qatar.tamu.edu}
\thanks{Ali Tajer is with the Department of Electrical, Computer, and Systems Engineering, Rensselaer Polytechnic Institute, USA, Email: tajer@ecse.rpi.edu.}
 \thanks{Khalid Qaraqe is with the Department
of Electrical and Computer Engineering, Texas A\&M University at Qatar, PO Box 23874, Education City, Doha, Qatar. Email: khalid.qaraqe@qatar.tamu.edu.}

}

\maketitle


\begin{abstract}
In order to foster renewable energy integration, improve power quality and reliability, and reduce hydrocarbon emissions, there is a strong need to deploy energy storage systems (ESSs), which can provide a control medium for peak hour utility operations. ESSs are especially desirable at the residential level, as this sector has the most untapped demand response potential. However, considering their high acquisition, operation, and maintenance costs, isolated deployment of ESSs is not economically viable. Hence, this paper proposes a \emph{sharing-based} ESS architecture, in which the demand of each customer is modeled stochastically and the aggregate demand is accommodated by a combination of power drawn from the grid and the storage unit when the demand exceeds grid capacity. The optimal size of energy storage systems is analyzed an analytical method is developed for a group customers with \emph{single} type of appliances. This framework is also extended to any network size with arbitrary number of customers and appliance types, where the analytical method provides a tractable solution to the ESS sizing problem. Finally, a detailed cost-benefit analysis is provided, where the results indicate that sharing-based ESSs are practical and yield significant savings in terms of ESS size.
\end{abstract}

\IEEEpeerreviewmaketitle
\section{Introduction}

Over the past few years, the power grids have become more stressed due to the steady increase in peak demand and more fragile due to the integration of intermittent renewable energy resources. Such a negative trajectory poses a profound threat to power grids, leads to a rise in hydrocarbon emissions, and exposes new economic challenges. One effective solution to mitigate the aforementioned issues, is to deploy energy storage systems (ESSs), which can act as an energy buffer and decouple the time of generation and demand by storing cheaper and cleaner off-peak hour electricity and delivering it during the peak load periods \cite{EPRI}. Integration of such storage units has multi-faceted monetary benefits for different players. For instance, end-users can enjoy the reduced electricity cost and lessened interruptions, while utility operators can enjoy improved grid reliability and asset utilization, and possibly deferral of system upgrades. Moreover, policymakers benefit from reduced hydrocarbon emissions and the mainstream integration of renewables, which are closely linked to energy security \cite{EPRI} and \cite{sandiaReport}. 

While employing storage devices at microgrids has certain benefits, their deployment based on the current technologies may still not be economically viable \cite{sandiaReport}. Therefore, optimal sizing of the storage units based on the realistic needs of the grids is a critical step towards the efficient operation of the grid. Specifically, over-provisioning ESS size entails costly and underutilized assets, whereas under-provisioning reduces its operating lifetime (e.g., frequently exceeding the allowable depth of charge level degrades its health). Hence, there is a strong need to develop analytical models to solve the sizing problem. Furthermore, the sizing requirements for storage units are usually application-dependent. For instance, energy arbitrage applications for independent system operators require 100s mega watts (MW) in size, while storage units employed for energy management for communities or microgrids require 10s kilo watts (kW). In this study, our main focus is on the latter one, as the residential sector has a great potential for peak demand reduction in the U.S. \cite{federal2009national}.

ESS sizing has received some attention in the literature. The work in~\cite{oudalov2007sizing} presents a sizing approach for {\em single} industrial customers for peak saving applications. The sizing problem is solved by maximizing the net benefit, which is the sum of reductions in the electricity bills minus the operation costs and the one-time acquisition cost. Similarly, \cite{sizingNAN} proposes a sizing framework using similar cost models for a microgrid, but it also considers savings due to the storage of energy generated from renewable resources. The work in \cite{solarShaping} develops a sizing approach based on stochastic network calculus to size the storage units coupled with solar power generation. It employs \emph{loss of load probability} as the main performance metric to provision the resources. Another probabilistic approach is presented in \cite{zaragozaSizing}, which couples forecasting errors in wind generation with the storage unit and sizes the storage unit according to a desired level of accuracy. Moreover, \cite{IBMRes} solves the sizing problem by using stochastic dynamic programming to minimize the operation cost and employs the storage device for balancing the load in a wind farm application. From the power engineering point of view, the sizing problem is usually solved via simulations \cite{lukicESS,sizing1,sizing2}. However, simulations techniques are usually computationally expensive, and their accuracy depends upon the availability of data traces.

In this paper, we develop an analytical framework for optimal energy storage sizing for sharing-based end-user applications. The proposed framework contributes to the existing literature in the following ways:

\begin{itemize}
\item The existing analytical methods for storage sizing focus on settings with {\em one} customer, or are developed for renewable integration and do not explore the underlying user dynamics. The proposed framework can cope with any network with arbitrary number of consumers with different levels of demand. Our analytical results show that a sharing-based design of storage units exhibits substantial gains over user-level designs.
\item The existing methods for multi-consumer settings are simulation-based. Lack of closed form expressions for the storage size-cost model limits the applicability of the model in optimizing the operations. The advantage of the proposed analytical method is hat it establishes the exact optimal sizing, and subsequently, is computationally less expensive. 
\item Majority of the existing literature assumes that residential users employ storage units for energy management applications without considering the associated costs. However, the studies conducted by Electric Power Research Institute \cite{EPRI} and Sandia Laboratories \cite{sandiaReport} reveal that the cost of employing stand-alone ESS is very high\footnote{The cost/benefit regime depends on different assumptions such as availability of solar rooftops, energy consumption statistics, project duration, and utility tariffs, to name a few. See references for details}. Hence, we provide a detailed economic analysis for ESS applications in residential usage and show that storage technologies can be economically viable, only if operated in a shared manner.
\end{itemize}

Proper sizing of storage units is ultimately linked to the customer electricity consumption model. Several measurement-based studies \cite{rosenberg1, onOffModels2,Vardakas2014455}, and \cite{Richardson} show that electricity consumption of households can be represented by On-Off models. For instance, the work presented in \cite{rosenberg1} measures the electricity consumption of $20$ households at fine-grained intervals (every $6$ second) and models the loads with continuous-time Markov chains. Driven by this observation, we adopt a Markovian fluid model to represent consumer's demands, and the analyzes rely on the stochastic theory of fluid dynamics. We establish the interplay among the minimum amount of storage size, the grid capacity, number of consumers, and the stochastic guarantees on outage events. 

We note that studying storage units at a network level is of paramount significance as they are expected to become integral to smart energy grids. More specifically, ESSs will be employed at smart residential, business complexes, and university campuses, to name a few, to reduce peak hour consumption. Clearly, in such sharing-based applications, the size of the energy storage is linked to the customer population, the power drawn from the grid, and the load profile. This relation is highly non-linear due to multiplexing gains which are computed by the percentage of reduction in the required amount of resources with respect to the baseline case of assigning peak demand to each user. 

\begin{table}[t]
 \begin{scriptsize}
  \centering
  \caption{Notations}
    \begin{tabular}{p{0.99cm}  p{6.9cm} }
    \toprule
    Parameter & Description \\
    \midrule
    $C_{t}$     & Power drawn from grid at time $t$. \\
    $N$ & Number of users in the single class case.\\
    $\bm{N}$     & Vector of the number of users for $K$ classes. Note that this is not the number of houses since in one house there can be multiple appliances requesting demand. \\
    $R_k$     & Demand of a user of class $k$.\\
    $\lambda_{k}$ & Arrival rate of charge request for class $k$. \\
    $\mu_{k}$ & Mean service rate for the customer demand of class $k$.\\
    $B$  & Size of the energy storage unit.\\
    $S(t)$  & ESS depletion level, $0$$\leq$$S(t)$$\leq$$B$. \\
    $L_{\bm{n}}(t)$ & Aggregated load on the system when $\bm{n}$ users are On. For single class index $i$ is used.\\
     ${F_i}(x)$& Steady state cumulative probability distribution function of ESS charge level. \\
    $\varsigma$&Grid power allocated per source ($C$/$N$).\\
    $\kappa$ & ESS per user ($B$/$N$).\\
    $\omega_{k}$ & \emph{Effective demand} for customer class $k$.\\
    $\zeta$ &Bustiness parameter.\\
    $\eta$ & Round-trip efficiency of the storage unit.\\
    \bottomrule
    \end{tabular}
    
  \label{summary}%
  \end{scriptsize}
\vspace{-15pt}
\end{table}%

\section{System Description}
We consider a community of consumers in which the demands of $N$ users are accommodated by the power grid capacity along with a {\em shared} energy storage system unit, in which capacity fluctuates over time. The grid capacity at time $t$ is denoted by $C_t$ and ESS depletion level is $S(t)$, for $t\in\mathds{R}^+$. Furthermore, the energy storage unit has the following parameters: 
\begin{itemize}
\item \emph{Energy rating} or the size of the storage is denoted by $B$ (kWh).

\item \emph{Power rating} is the rate at which storage can be charged or discharged. The charging rate is assumed to be $P_{c}\leq C_{t}, \forall t$, and discharge rating $P_{d}$ is related to the energy rating $B$ since $B=P_{d}\times\mbox{(desired support duration)}$. Desired support duration is typically equivalent to the length of the peak hour.
\item The \emph{efficiency} of a charge-discharge cycle is modeled by $\eta  \in \left[ {0,1} \right]$ to capture the percentage of stored energy or the fraction of energy that is transferred after losses are extracted. 
\item \emph{Dissipation losses} represent a small percentage of losses that occur due to leakage. For simplicity in notations, dissipation losses are ignored.
\item For economic analysis, we follow the standard assumptions made in \cite{EPRI} and assume that the project will be sited for $15$ years, with $10$\% discount rate and the economic analysis is carried out by computing the net present value of the storage in the first year by using the cost per kW or kWh.
\end{itemize}

Furthermore, since the consumers do not necessarily have identical demands, we consider $K$ customer classes, which are distinguished by the amount of electricity demands,  represented by $\{R_{k}\}$, where $R_{k}$ denotes the energy demand per time unit for customer type $k$. Let $N_k$ denote the number of consumers of type $k$ and the vector $\bm{N}$ represents the number of consumers of each type, that is $\bm{N}=(N_1,\cdots,N_K)$. As established in~\cite{jsac,TSG15,rosenberg1,onOffModels2} and \cite{Richardson} the consumption pattern of each consumer can be well-represented by a two-state ``On/Off'' process. We define the binary variable $s^{ik}_t$ to represent the state of consumer $i$ of type $k$ at time $t$ such that
\begin{equation}
s^{ik}_t=\left\{
\begin{array}{ll}
1 & \mbox{consumer $i$ is On}\\
0 & \mbox{consumer $i$ is Off}
\end{array}\right. \ .
\end{equation}
When a customer is in the ``On'' state, it initiates an energy demand. The duration of demand is modeled statistically, which is adopted to capture the variety types of consumers' demands. Specifically, the duration of type $k$ customer's demand is assumed to be exponentially distributed with parameter $\mu_k$. Furthermore, we assume that the requests, that are transitions from ``Off'' to ``On'', are generated randomly and according to a Poisson process with parameter $\lambda_k$.  Hence,
for each consumer $i$ of type $k$ at any time $t$ we have
\begin{equation}
\mathds{P}(s^{ik}_t\;=\;1)\;=\; \frac{\lambda_k}{\lambda_k+\mu_k}\ .
\end{equation}In order to formalize the dynamics of ESS, we define $L_{ik}(t)$ as the charge request of consumer $i$ of type $k$. Then, there are exactly three cases that define the rate of change in the storage unit: (i) the storage can be in the fully charged state, and the aggregate demand is less than grid power $C_{t}$; (ii) the storage can be completely discharged, and the aggregate demand is more than the grid power $C_{t}$; and (iii) any partially charged state with any customer demand. Therefore, for the rate of change in the storage level of the ESS, we have
{\small
\begin{equation}
\frac{{d{S}(t)}}{{dt}} = \left\{ {\begin{array}{ll}
0 , &\mbox{if}\;{S}(t)=B\;\mbox{\&}\;\\
& \sum\limits_{k=1}^K\sum\limits_{i=1}^{N_{k}} \eta L_{ik}(t)<C_t\vspace{.1 in}\\[-2mm]
0 ,&\mbox{if}\;{S}(t)=0\;\mbox{\&}\;\\ 
&\sum\limits_{k=1}^K\sum\limits_{i=1}^{N_{k}}  \eta L_{ik}(t) >C_t\vspace{.1 in}\\[-2mm]
{ {\eta({C_t-{\sum\limits_k\sum\limits_i L_{ik}}}(t)) }},&\mbox{otherwise}
\end{array}} \right. \ .
\end{equation}
}
Due to stochasticities involved (in consumption and generation), by choosing any storage capacity $B$, only stochastic guarantees can be provided for the reliability of the system and always there exists a chance of outage, which occurs when available resources fall below the aggregate demands by the consumers. By noting that $S(t)$ denotes the energy level that the storage unit needs to feed into the grid to avoid outage, an outage event occurs when the necessary load from the storage unit exceeds the maximum available $B$. Hence, we define $\varepsilon$-outage storage capacity, denoted by $B(\varepsilon)$, as the smallest choice of $B$ corresponding to which the probability of outage does not exceed $\varepsilon\in(0,1)$, i.e.,
\begin{equation}\label{eq:problem}
B(\varepsilon)=\left\{
\begin{array}{ll}
\min & B\vspace{.1 in}\\[-3mm]
{\rm s.t.} & \mathds{P}\big(S_t\geq B\big)\leq \varepsilon
\end{array}\right. \ .
\end{equation}Our goal is to determine the $\varepsilon$-outage storage capacity $B(\varepsilon)$ based on grid capacity $C_t$, the number of users $\bm{N}$, and their associated consumption dynamics. For the simplicity of mathematical expressions, we scale the storage parameters and instead of $B/\eta$ we redefine $B$ as the maximum amount of energy that can be stored, and similarly $P_{d}$ and $P_{c}$ represent the actual power ratings. The notations are summarized in Table \ref{summary}. We start our analysis for deriving $B(\varepsilon)$ for a single class case.

\section{Storage Capacity Analysis For Single Class Customers ($K=1$)}

\subsection{Storage Access Dynamics}
When the grid can serve all the consumers' demands, there will be no consumer served by the storage unit. On the other hand, when the grid capacity falls below the aggregate demand, the consumers access the storage unit. Since the requests of the consumers arrive randomly, the number of consumers accessing the unit also varies randomly.

Since we have $N$ {\em independent} consumers each with a two-state model, by taking into account their underlying arrival and consumption processes, the composite model counting the number of users accessing the storage unit at a given time can be modeled as a {\em continuous}-time birth-death process. Specifically, this process consists of $(N+1)$ states, in which state $n\in\{0,\dots,N\}$ models $n$ consumers being active and accessing the storage unit, i.e.,
\begin{equation}
\mbox{state at time $t$ is $n$ if}\quad \sum_{i=1}^Ns^i_t=n\ ,
\end{equation}and drawing $nR_p$ units of power from the storage unit. As depicted in Fig.\ref{MMPP}, the transition rate from state $n$ to state $n+1$ is $(N-n)\lambda$ and, conversely, the transition from state $n+1$ to state $n$ is $(n+1)\mu$. Hence, for the  associated infinitesimal generator matrix $M$, in which the row elements sum to zero, for $i,n\in\{0,\dots,N\}$ we have

{\small
\begin{equation}\label{matrix}
M[i,n]=\left\{
\begin{array}{ll}
-((N-i)\lambda + i\mu) & n=i \vspace{.1 in}\\[-2.5mm]
i\mu & n=i-1\;\;\&\;\; i> 0\vspace{.1 in}\\[-2.5mm]
(N-i)\lambda & n=i+1\;\;\&\;\; i< N\vspace{.1 in}\\[-2.5mm]
0 & \mbox{otherwise}
\end{array}\right. \ .
\end{equation}
}
By denoting the stationary probabilities of state $n\in\{0,\dots,N\}$ by $\pi_n$ and according defining $\bm{\pi}  = \left[ {{\pi _0},{\pi _1},...,{\pi _N}} \right]$, these stationary probability values satisfy $\bm{\pi} M=\bzero$.

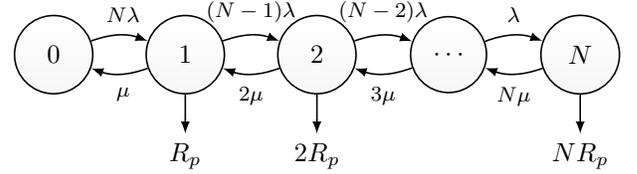
\begin{figure}[t]
\centering
\begin{tikzpicture}[-latex ,auto ,node distance =1.6 cm and 1.75cm ,on grid ,
semithick ,
state/.style ={ circle ,top color =white , bottom color = processblue!20 ,
draw,black , text=black , align=center,text width =0.7 cm, minimum size=10pt}]
\node[state] (C){$0$};
\node[state]       (v1)  [right =of C]   {$1$};
\node[state]       (v2)  [right =of v1]   {$2$};
\node[state]       (v3)  [right =of v2]   {$\cdots$};
\node[state]       (v4)  [right =of v3]   {$N$};
\node[draw=none,below= 1.35cm of v1]  (v1x) {${R_p} $};
\node[draw=none,below= 1.35cm of v2]  (v2x) {${2R_p} $};
\node[draw=none,below= 1.35cm of v4]  (v4x) {${NR_p} $};
\path (C) edge [bend left =20] node[above =0.05 cm, pos=0.5]  {\begin{footnotesize} $N\lambda$\end{footnotesize}  }  (v1);
\path (v1) edge [bend right = -20] node[below =0.05 cm, pos=0.5] {\begin{footnotesize} $\mu$\end{footnotesize}  }(C);
\path (v1) edge [bend left =20] node[above =0.05 cm, pos=0.5] {\begin{footnotesize}$(N-1)\lambda$\end{footnotesize}}  (v2);
\path (v2) edge [bend right = -20] node[below =0.05 cm, pos=0.5] {\begin{footnotesize}$2\mu$\end{footnotesize}}  (v1);
\path (v2) edge [bend left =20] node[above =0.05 cm, pos=0.5] {\begin{footnotesize}$(N-2)\lambda$\end{footnotesize}} (v3);
\path (v3) edge [bend right = -20] node[below =0.05 cm, pos=0.5] {\begin{footnotesize}$3\mu$\end{footnotesize}} (v2);
\path (v3) edge [bend left =20] node[above =0.05 cm, pos=0.5] {\begin{footnotesize}$\lambda$\end{footnotesize}} (v4);
\path (v4) edge [bend right = -20] node[below =0.05 cm, pos=0.5] {\begin{footnotesize}$N\mu$\end{footnotesize}} (v3);
\path    (v1) edge       node[left = 0.05 cm] {} (v1x);
\path    (v2) edge       node[left = 0.05 cm] {} (v2x);
\path    (v4) edge       node[left = 0.05 cm] {} (v4x);

\end{tikzpicture}

\caption{Composite model for $N$ independent users for single user type (k=1). Each user becomes active (``On'') at rate $\lambda$ and becomes inactive (``Off'') at rate $\mu$. The aggregate demand depends on the active number of users. }\label{MMPP}
\vspace{-15pt}
\end{figure}

\subsection{Analyzing Distributions}\label{distribution}

Given the dynamics of accessing the storage unit, in the next step we analyze the statistical behavior of the ESS charge level. Specifically, we define ${F_i}(t,x)$ as the cumulative distribution function (cdf) of the ESS charge level when $i\in\{0,\dots,N\}$ consumers are depleting the storage unit at time $t$, i.e.,
\begin{equation}\label{eq:cdf}
{F_i}(t,x)\;=\; \mathds{P}\Big (S(t) \leq x \;\;\;\mbox{and}\;\;\; \sum_{j=1}^Ns^j_t=i \Big)\ .
\end{equation}
Accordingly, we define the vector of cdfs as
\begin{equation}
\boldsymbol{F}(t,x)\triangleq\left[ {{F_0}(t,x)\; , \; {F_1}(t,x)\; , \;... \; , \;  {F_N}(t,x)} \right]\ .
\end{equation}
Based on this definition, the next lemma delineates a differential equation which admits the cdf vector as its solution and is instrumental for analyzing the probability of outage events, i.e.,  $\mathds{P}\big(\sum_{i=}^NL_i(t)\;>\; C_t+B\big)$.
\begin{lem}
The cdf vector ${\boldsymbol F}(t,x)$ satisfies
\begin{equation}\label{diffEq}
\frac{d{\boldsymbol F}(t,x)}{dx}\cdot D = {\boldsymbol F}(t,x)\cdot M\ ,
\end{equation}
where $D$ is a diagonal matrix defined as
\begin{equation}\label{D}
D\triangleq \mbox{\rm diag}\left[ { - C_t\mu\; , \; (1 - C_t)\mu\; , \; ... \; , (N - C_t)\mu } \right]\ ,
\end{equation}
and matrix $M$ is defined in \eqref{matrix}.
\end{lem}
\begin{proof}

In order to compute the probability density functions, we find the expansion of $F_i(t,x)$ for an incremental change $\Delta t$ in $t$, i.e, ${F_i}(t + \Delta t,x)$. Note that during incremental time $\Delta t$, three elementary events can occur:
\begin{enumerate}
\item one inactive consumer might become active, i.e., $i$ increases to $i+1$;
\item one active consumer becomes inactive,  i.e., $i$ reduces to $i-1$; or
\item the number of active consumers remains unchanged.
\end{enumerate}
Since the durations of arrival and departure of consumers are exponentially distributed as cdf $F_i(t,x)$ can be expanded
\begin{align}\label{generating}
\nonumber {F_i} & (t + \Delta t,x) \\ \nonumber
&= \underbrace{\left[ {N - (i - 1)} \right]\cdot  (\lambda \Delta t)\cdot {F_{i - 1}}(t,x)}_{\text{one consumer added}}\\[-2mm]
\nonumber & + \underbrace{[i + 1]\cdot(\mu \Delta t)\cdot {F_{i + 1}}(t,x)}_{\text{one consumer removed}} \\[-2mm]
\nonumber &+ \underbrace{\left[ {1 - \left( {(N - i)\lambda  + i\mu } \right)\Delta t} \right]\cdot {F_i}(t,x - (i - C_t)\cdot \mu \Delta t)}_{\text{no change}}\\[-2mm]
&  + o\left( {\Delta t^2} \right),
\end{align}
where $ o\left( {\Delta t^2} \right)$ represents the probabilities of the compound events and tends to zero more rapidly than ${\Delta t^2}$ (and $\Delta t$) as $\Delta t \to 0$. Next, by passing the limit
$$\mathop {\lim }\limits_{\Delta t \to 0} \frac{{{F_i}(t + \Delta t,x)}}{{\Delta t}}$$
it can be readily verified that (\ref{generating}) simplifies to
\begin{align}\label{derives}
\nonumber  \frac{{\partial {F_i}(x,t)}}{{\partial t}} & = \left[ {N - (i - 1)} \right]\cdot  (\lambda)\cdot {F_{i - 1}}(t,x)\\
\nonumber & + [i + 1]\cdot(\mu)\cdot {F_{i + 1}}(t,x)\\
\nonumber & - \left[ {(N - i)\lambda  + i\mu } \right]\cdot {F_i}(t,x)\\
 & - (i - C_t)\cdot(\mu)\cdot \frac{{\partial F_i(t,x)}}{{\partial x}}\ ,
\end{align}
where we have defined ${F_{ - 1}}(t,x)= {F_{N + 1}}(t,x) =0$. By recalling that the design is intended gor a long-term steady-state operation we have ${{\partial {F_i}(x,t)} \mathord{\left/
 {\vphantom {{\partial {F_i}(x,t)} {\partial t = 0}}} \right.
 \kern-\nulldelimiterspace} {\partial t = 0}}$. Hence, (\ref{derives}) can be rewritten as
\begin{align}\label{recursion}
\nonumber (i - C_t)\cdot(\mu)\cdot \frac{{\partial F_i(t,x)}}{{\partial x}} & = \left[ {N - (i - 1)} \right]\cdot  (\lambda)\cdot {F_{i - 1}}(t,x)\\ \nonumber
\nonumber & + [i + 1]\cdot(\mu)\cdot {F_{i + 1}}(t,x)\\
 & - \left[ {(N - i)\lambda  + i\mu } \right]\cdot {F_i}(t,x)\ .
\end{align}
By concatenating all the equations \eqref{recursion} for all $i\in\{0,\dots,N\}$ we obtain the compact form

\begin{equation}\label{diffEq}
\frac{d{\boldsymbol F}(t,x)}{dx}\cdot D = {\boldsymbol F}(t,x)\cdot M\ .
\end{equation}
\end{proof}The solution of the first order differential equation given in \eqref{diffEq} can be expressed as a sum of exponential terms. The general solution requires computing $(N+1)$ eigenvalues of the matrix $MD^{-1}$ and the general solution can be expressed as~\cite{anick1}:
\begin{equation}\label{diff2}
\boldsymbol{F}(t,x) = \sum_{i = 0}^N {{\alpha}_i}\;\bm{ \phi}_i\;\exp(z_ix)\ ,
\end{equation}
where $z_i$ is the $i^{th}$ eigenvalue of $MD^{-1}$ with the associated eigenvector $\bm{ \phi} _i$, which satisfy ${z_i}\bm{\phi }_iD = \bm{\phi}_iM$. The coefficients $\{\alpha_0,\dots,\alpha_N\}$ are determined by the boundary conditions, e.g., $F_i(t,0)=0$ and $F_i(t,\infty)=1$.

In order to compute the probability distribution in \eqref{diff2}, we need to determine the eigenvalues of $MD^{{-1}}$, the eigenvectors $\{\bm {\phi}_i\}$, and coefficients $\{\alpha_i\}$.
Note that, since $x \geq 0$ and $F_{j}(t,x)$ is upper bounded by $1$, all of the positive eigenvalues and the corresponding $\alpha_i$ must be set to zero, hence this reduces the computational complexity and \eqref{diff2} simplifies to:
\begin{equation}\label{diff4}
\boldsymbol{F}(t,x) = \sum_{i:Re[z_{i}\leq0]} {{\alpha}_i}\;{\bm \phi_i}\;\exp(z_ix)\ .
\end{equation}It can be further observed that since ${z_i}\bm{\phi _i}D = \bm{\phi _i}M$, one of the eigenvalues must be zero. Then by setting $z_{0}=0$, the corresponding eigenvector can be computed from $\bm{\phi}_0 M=\bm{0}$. We also earlier showed  that the steady state probability distribution $\bm{\pi}$ of the $N+1$ state Markov chain can be computed from the same equation, that is $\bm{\pi} M=\bzero$. Since, the eigenvector $\bm{\phi}_0$ is known and one of the eigenvalues is $z_{0}=0$, we can write $\bm{\phi} _0=\bm{\pi}$. Therefore, \eqref{diff4} further simplifies to \cite{schwartz1996}:
\begin{equation}\label{diff5}
\boldsymbol{F}(t,x) = \boldsymbol{\pi} +\sum_{i:Re[z_{i}< 0]} {{\alpha}_i}\;\bm{\phi}_i\;\exp(z_ix)\ .
\end{equation}


\subsection{Single User Storage Capacity ($K=1$, $N=1$)}

For computing the desired $\varepsilon$-outage storage capacity $B(\varepsilon)$ by leveraging the cdf vector found in \eqref{diff5} we start by a simple network with one user ($N=1$). The insights gained can be leveraged to generalize the approach for networks with any arbitrary size $N$. When $N=1$ the infinitesimal generator matrix $M$ defined in \eqref{matrix} is
\begin{equation}
M\;=\;\begin{bmatrix}
  - \lambda     & \lambda \\
\mu & -\mu
\end{bmatrix}\ .
\end{equation}
For finding the expansion of $\boldsymbol{F}(t,x)$ as given in \eqref{diff2} we need to find the eigenvalues of $MD^{-1}$, i.e. $z_{0}$ and $z_{1}$, where $D$ is defined in \eqref{D}. Based on \eqref{D} we find that
\begin{equation}
MD^{-1}\;=\;
\begin{bmatrix}
  \frac{1}{C_t}\cdot \frac{\lambda}{\mu}     & -\frac{1}{1-C_t}\cdot \frac{\lambda}{\mu} \\
-\frac{1}{C_t} & -\frac{1}{1-C_t}
\end{bmatrix}\ .
\end{equation}
Hence, the eigenvalues are
\begin{equation}\label{eigenValue}
z_0=0\quad\mbox{and}\quad z_1=\frac{\chi}{C_t}-\frac{1}{{1 - C_t}},
\end{equation}
where we have defined $\chi\triangleq \frac{\lambda}{\mu}$. It can be readily verified that the eigenvector associated with $z_1$ is $\bm{\phi}_1=[1-C_t\;,\; C_t]$. Therefore, according to \eqref{diff5} we have
\begin{align}\label{diff6}
\boldsymbol{F}(t,x) = \boldsymbol{\pi} +{\alpha}_1\;\bm{\phi}_1\;\exp(z_1x)\ .
\end{align}
Finally, by finding the coefficient $\alpha_1$ we can fully characterize $\boldsymbol{F}(t,x)$. This can be facilitated by leveraging the boundary condition $F_1(t,0)=0$, which yields
\begin{equation}
F_1(t,0)\;= \; \pi_1 \; + \; \alpha_1\;C_t\;=\; 0 \ ,
\end{equation}
where we have that $\pi_1=\frac{\lambda}{\lambda+\mu}$. Therefore
\begin{equation}\label{alphaValue}
\alpha_1\;=\; -\frac{\chi}{C_t(1+\chi)}\ ,
\end{equation}
which, subsequently, fully characterizes both cdfs $F_0(t,x)$ and $F_1(t,x)$ according to
\begin{eqnarray*}
F_{0}(t,x) & = & \pi_{0}+\alpha_1(1-C_t)\exp(z_1x),\\
\mbox{and}\quad F_{1}(t,x) & = & \pi_{1}+\alpha_1 C_t\exp(z_1x)\ .
\end{eqnarray*}
As a result, by recalling the definition of $F_i(t,x)$ in \eqref{eq:cdf}, the probability that the storage level $S_t$ falls below a target level $x$ is given by
\begin{equation}\label{mainResult}
\mathds{P}(S_t \leq x)\; = \;  {F_0}(x) + {F_1}(x) = 1 + \alpha_1 \exp(z_1x)\ .
\end{equation}
Given this closed-form characterization for the distribution of $S_t$, we can now evaluate the probability term
\begin{equation}\label{eq:P}
\mathds{P}(S_t>B)\ ,
\end{equation}
which is the core constraint in the storage sizing problem formalized in \eqref{eq:problem}. Specifically, for any instantaneous realization of $C_t$  denoted by $c$ we have
\begin{align}\label{eq:St}
\nonumber \mathds{P}(S_t>B) & = \int_{C_t}\mathds{P}(S_t>B\;|\; C_t=c)\;f_{C_t}(c)\;dc\\
\nonumber & =-\int_{C_t}\alpha_1\exp(z_1B)\;f_{C_t}(c)\;dc\\
\nonumber & =\int_{C_t}\frac{\chi}{c(1+\chi)}\exp\left( \frac{B\chi}{c}-\frac{B}{1 -c}\right)f_{C_t}(c)\;dc\ .
\end{align}
Therefore, by noting that $z_1=\frac{\chi}{c}-\frac{1}{1-c}$ is negative, the probability term $\mathds{P}(S_t>B)$ becomes strictly decreasing in $B$. Hence, the smallest storage capacity $B$ that satisfies the stochastic guarantee $\mathds{P}(S_t>B)\leq \varepsilon$ has a unique solution corresponding to which this constraint holds with equality. In the simplest settings in which grid capacity $C_t$ is constant $c$ we find
\begin{align}
B(\varepsilon) &= \frac{c(1-c)}{\chi - \chi c-c} \cdot \log\frac{\varepsilon c(1+\chi)}{\chi}\ .
\end{align}
\subsection{Multiuser Storage Capacity ($K=1$, $N>1$)}

In this subsection we provide a closed-form for the probability term $\mathds{P}(S_t \leq x)$ for arbitrary values of $N$, which we denote by $F_{N}(x)$. Computing all $F_N(x)$ terms through computing their constituent terms $F_i(t,x)$, especially as $N$ grows, becomes computationally expensive, and possibly prohibitive as it involves computing the eigenvalues and eigenvectors of $MD^{{-1}}$. By capitalizing on the observation that for large number of users $N\gg 1$, the largest eigenvalues are the main contributors to the probability distribution,~\cite{morrison1989} shows that, the asymptotic expression for $F_N(x)$ is given by
\begin{align}\label{Nusers}
F_{N}(x) & = \frac{1}{2}\sqrt {\frac{u}{{\pi f(\varsigma )(\varsigma  + \lambda (1 - \varsigma ))N}}}\\[-1mm] \nonumber
& \;\;\times \exp(- N\varphi (\varsigma)- g(\varsigma )x)\\[-1mm]\nonumber
& \;\;\times \exp(- 2\sqrt {\left\{ {f(\varsigma )(\varsigma  + \lambda (1 - \varsigma ))Nx} \right\}}),\\[-3mm]\nonumber
\end{align}
where,
{\small
\begin{align*}
f(\varsigma ) & \triangleq \log \left( {\frac{\varsigma }{{\lambda (1 - \varsigma )}}} \right) - 2\frac{{\varsigma (1 + \lambda ) - \lambda }}{{\varsigma  + \lambda (1 - \varsigma )}\ },\\ 
u &  \triangleq  \frac{{\varsigma (1 + \lambda ) - \lambda }}{{\varsigma (1 - \lambda )}}\ ,\\ 
\varphi (\varsigma ) & \triangleq \varsigma \log (\varsigma ) + (1 - \varsigma )\log (1 - \varsigma ) - \varsigma \log (\varsigma ) + \log (1 + \lambda )\ ,\\[-1mm]
g(\varsigma ) & \triangleq z + 0.5\left( {\varsigma  + \lambda (1 - \varsigma )} \right)\frac{{\psi (1 - \varsigma )}}{{f(\varsigma )}}\ ,\\[-1mm]
z & \triangleq (1 - \lambda ) + \frac{{\lambda (1 - 2\varsigma )}}{{(\varsigma  + \lambda (1 - \varsigma ))}}\ ,\\[-1mm]
\mbox{and}\;\psi  & \triangleq \frac{{(2\varsigma  - 1){{(\varsigma (1 + \lambda ) - \lambda )}^3}}}{{\varsigma {{(1 - \varsigma )}^2}{{(\varsigma  + \lambda (1 - \varsigma ))}^3}}}\ .
\end{align*}
}
In this set of equations, time is measured in units of a single average ``On'' time ($1/\mu$). Furthermore, $\kappa$ and $\varsigma$ are defined as the ESS per user ($B/N$) and the grid power allocated per one source, respectively. Furthermore, we denote the variable $\upsilon $ as the power above the mean demand allocated per user as $\upsilon  = \varsigma  - \frac{\lambda }{{1 + \lambda }}$.

\section{Storage Capacity Analysis For Multi-Class Customers ($K>1$, $N>1$)}
In this section, we consider the case where the storage serves more than one customer type. We start our analysis by noting that the continuous-time Markov chain model presented in Fig. \ref{MMPP}, becomes a $K$ dimensional Markov process. Recall that total number of users of each class is represented by $\bm{N}=(N_1,\cdots,N_K)$ and we let $\bm{n}(t)=[n_1(t)\dots n_k(t) ]$ represent the number of sources of type $k$ that are On at time $t$.  Furthermore, $\bm{n}(t)$ represents the state space of the Markov process and similar to single class case given in \eqref{matrix}, the transitions can occur between neighboring states, and the transition matrix  $ \bar{M} = \left\{  {\bar M(i,n)} \right\}$ is given by,
\begin{align}
\bar M(\bm{n},\Delta _k^ + (\bm{n})) &= ({N_k} - {n_k}){\lambda _k},\label{trans1}\\ 
\bar M(\bm{n},\Delta _k^ - (\bm{n})) &= {n_k}{\mu _k},\label{trans2}\\
\bar M(\bm{n},\bm{n}') &= 0, \mbox{otherwise.}\label{trans3}
\end{align}
where
\[\begin{array}{l}
\Delta _k^ + ({n_1},\cdots,{n_k},\cdots,{n_K}) = ({n_1},\cdots,{n_k} + 1,\cdots,{n_K})\\
\Delta _k^ - ({n_1},\cdots,{n_k},\cdots,{n_K}) = ({n_1},\cdots,{n_k} - 1,\cdots,{n_K}).
\end{array}\]
In the single class case, the transitions can only occur between $2$ neighboring states. In the multi-class case, however, the transitions can occur in $2\times k$ different neighbors. For instance, assume that there are two classes, $k=2$, and the Markov chain is in state $(1,1)$ (i.e., one active user from each class). Therefore, there exists four possible transitions $(1,0), (0,1), (2,1)$ and $(2,2)$, and the transition rates depend on the $(\lambda_{k},\mu_{k})$ pairs. To that end, the transition rates given in \eqref{trans1} represent customer arrivals, whereas the transitions rates in \eqref{trans2} account for new customer arrivals, and similar to single class case there can not be any transitions beyond neighboring states. This is given in \eqref{trans3}. 
\subsection{Effective Demand}

The analysis in Section \ref{distribution} provided for the single-class settings is also valid for the multi-class case. However, computing the eigenvalues becomes infeasible when $K>1$. Hence, we follow the decomposition method proposed in \cite{gibbens} to efficiently evaluate the eigenvalues $z_{i}$ and the coefficients $\alpha_{i}$. Our goal in this section is to compute $\omega_{k}$ which is a deterministic quantity acting as a surrogate for the actual aggregate stochastic demand \cite{kelly1991effective}. In other words, assigning $\omega_{k}$ amount of resources to each customer class will satisfy \eqref{eq:problem}.

We define parameter $\xi= 1 - \varepsilon ^{{1/B}}  \in \left[ {0,1} \right]$ to account for the susceptibility of the storage to customer demand pattern. This parameter will be an essential part of $\omega_{k}$. First we will consider large storage units, then we will show that the proposed scheme is also valid for any storage size. Assume that we are able to choose $\bm{N}=(N_1,\cdots,N_K)$ in a way that ensures $\mathds{P}\big(S_t>B\big)\leq \varepsilon$ for small $\varepsilon$ and the eigenvalues are relabeled such that ${z_0} \ge {z_1} \ge \dots \ge 0$, then the following theorem holds.

\begin{thm}\label{mainTheorem}

Let $\mathcal{B}(B,\varepsilon)=\{\bm{N}:\mathds{P}\big(S_t\geq B\big)\leq \varepsilon \}$. For large storage size $B$ and for small $\varepsilon$, we have,
\[\mathop {\lim }\limits_{B = \infty ,\varepsilon  \to 0} \frac{{\log \varepsilon }}{B} \to \zeta  \in \left[ { - \infty ,0} \right]\]Furthermore, let
\begin{equation}
\mathcal{\tilde{B}} = \left\{ {\bm{N}:\sum\limits_k {{\omega _k(\xi)}{N_k} < C} } \right\}, \nonumber 
\end{equation}
and,
\begin{equation}
 \mathcal{\bar{B}}= \left\{ {\bm{N}:\sum\limits_k {{\omega _k(\zeta)}{N_k} \le C} } \right\}, \nonumber
\end{equation}
where effective demand can be computed by \cite{gibbens},
\begin{equation}\label{effDemand}
{\omega _k}(\zeta ) = \frac{{\zeta R_{k} + {\mu _k} + {\lambda _k} - \sqrt {{{(\zeta R_{k} + {\mu _k} - {\lambda _k})}^2} + 4{\lambda _k}{\mu _k}} }}{{2\zeta }} \cdot
\end{equation}
Then, $\mathcal{\tilde{B}} \subseteq \mathcal{B}(B,\varepsilon)  \subseteq  \mathcal{\bar{B}}$.
\begin{proof}
For a given set of users, $\bm{N}$, the following holds.
\[
\frac{{\mathds{P}\big(S_t\geq B\big)}}{{\varepsilon}}= \frac{{\sum\nolimits_k {{\alpha _k}({\mathbf{1}^T}{\bm{\phi}_k}){\exp({{z_k}B})}} }}{\varepsilon },
\]
where $z_{k}$'s are the non-positive eigenvalues of $\bar{M}D^{-1}$, $\bm{\phi}_k $ are the corresponding eigenvectors, and $\mathbf{1}^T=[ {1,1,\cdots,1} ]$. Recall that, the coefficients $\alpha_{k}$  can be obtained by solving the following set of equations.

\[\left\{ \begin{array}{{rr}}
\sum\nolimits_k {\alpha _k}{\phi _k} = 0,&\;\mbox{if}\;{D_k} > 0,\\
{\alpha _k} = 0,&\;\mbox{if}\;{z_k} > 0,\\
{\alpha _0}{1^T}{\phi _0} = 1,&\;\mbox{otherwise.}
\end{array} \right.\]In\cite{kosten},
 it is shown that $z_{0}>z_{k}$, for $\forall k \ge 1$, and $z_{0}$ satisfies $C=\sum\nolimits_k {{\alpha _k}(z_{0})N_{k}}$. Then the following holds,
\begin{equation}\label{proof}
\frac{{{\mathds{P}\big(S_t\geq B\big)}}}{\varepsilon } = {\alpha _0}{\bm{1}^T}{\bm{\phi} _0}{\exp^{({z_0} - \zeta )B}}(1 + o(1)),
\end{equation} as $B\to\infty$. Differentiation shows that, effective demand $\omega_{k}(y)$ decreases as $y$ increases. Thus, for $\bm{N}\in \mathcal{B}$, $z_{0}<\zeta$ and $\frac{{{\mathds{P}\big(S_t\geq B\big)}}}{\varepsilon } \to 0$ as $B\to \infty$, which indicates that $\bm{N} \in \mathcal{B}(B,\varepsilon)$. In a similar manner, when $\bm{N} \notin \mathcal{\bar{ B}}$, then $z_{0}>\zeta$ and $\frac{{{\mathds{P}\big(S_t\geq B\big)}}}{\varepsilon } \to \infty$, and therefore $\bm{N}\notin \mathcal{B}(B,\varepsilon)$. This completes the proof.
\end{proof}
\end{thm}
The interpretation of Theorem \ref{mainTheorem} is that \eqref{eq:problem} holds if and only if $\sum\limits_k {{\omega _k(\zeta)}{N_k}} \le C$.
Furthermore, the following remarks hold.
\begin{remark}
The corresponding effective demand for $\zeta=0$ is $\omega_{k}(\zeta)=\frac{\lambda_{k}R_{k}}{\lambda_{k}+\mu_{k}}$, which is the mean customer demand for each class. Another interpretation of this result is that, before an outage event occurs at the storage, each customer turns ``On'' and ``Off'' their appliances many times that the aggregated demand equals to the mean customer demand. This result serves as a lower bound for the effective demand.
\end{remark}
\begin{remark}
At the other boundary, the corresponding effective demand for $\zeta=z$ is its peak demand $R_{k}$. In this case, each user is at ``On'' at all times, hence peak demand must be allocated.
\end{remark}Remarks given above suggest that the system is more susceptible to bustiness (longer ``On'' durations) for larger $\zeta$ values. Also, similar to the multi-class case, the largest contributor to the probability distribution is the largest eigenvalue, i.e., $z_{0}$.

\section{Economies of ESS Deployment}
\begin{table}[t]
  \centering
  \caption{Cost of poor power quality.\cite{reliability}}
    \begin{tabular}{rrr}
    \toprule
    \$/kW & Average & High \\
    \midrule
    Residential  & \$0.10 & \$0.60 \\
    Small C\&I & \$0.42 & \$2.52 \\
    Large C\&I & \$1.42 & \$14.00 \\
    \bottomrule
    \end{tabular}%
  \label{QCost}%
\end{table}%
\begin{table}[t]
  \centering
  \caption{Average outage cost \cite{PowerOutageRep}}
    \begin{tabular}{rrrrr}
    \toprule
    \$/kW & 15 Min. & 30 Min. & 1 Hour & 2 Hours \\
    \midrule
    Residential  & \$0.05  & \$0.60  & \$2.60  & \$3.95  \\
    Small C\&I & \$8.65  & \$16.01  & \$23.37  & \$48.91  \\
    Large C\&I & \$4.79  & \$7.46  & \$10.12  & \$17.96  \\
    \bottomrule
    \end{tabular}%
  \label{Ocost}%
\end{table}%
\begin{table}[t]
  \centering
  \caption{Average Time Of Use (TOU) rates in the United States 
  \cite{EPRI}}
    \begin{tabular}{rrrrr}
    \toprule
          & \multicolumn{2}{c}{Summer} & \multicolumn{2}{c}{Winter} \\
    \midrule
    \$/kWh & Peak  & Off-Peak & Peak  & Off-Peak \\
    \midrule
    Residential  & \$0.25  & \$0.06  & \$0.13  & \$0.06  \\
    Small C\&I & \$0.18  & \$0.05  & \$0.12  & \$0.05  \\
    Large C\&I & \$0.06  & \$0.04  & \$0.05  & \$0.04  \\
    \bottomrule
    \end{tabular}%
  \label{TOURates}%
\end{table}%
\begin{table}[t]
  \centering
  \caption{Demand Charges}
    \begin{tabular}{rrrrr}
    \toprule
          & \multicolumn{2}{c}{Summer} & \multicolumn{2}{c}{Winter} \\
    \midrule
    \$/kW-Month & Peak  & Off-Peak & Peak  & Off-Peak \\
    \midrule
    Residential  & \$0.00  & \$0.00  & \$0.00  & \$0.00  \\
    Small C\&I & \$15.00  & \$15.00  & \$8.00  & \$8.00  \\
    Large C\&I & \$12.00  & \$12.00  & \$10.00  & \$10.00  \\
    \bottomrule
    \end{tabular}%
  \label{DCost}%
\end{table}%

The proposed framework targets enhancing energy management via deploying for end users and  the monetary associated benefits of ESS deployment \cite{EPRI}.
\begin{enumerate}
\item \emph{Improved power quality} refers to voltage sags and outages experienced by customers, which in many cases remains unnoticed. However if occur for sufficiently large durations, they damage customers' appliances. The costs of momentary outages, according to a survey study conducted in \cite{reliability}, are summarized in Table \ref{QCost}. 
\item \emph{Improved power reliability}\label{ben1} refers to the usage of storage units during outages and blackouts. In order to quantify the cost of power reliability, we adopt the measure used in\cite{PowerOutageRep}, i.e., the cost per kW shown in Table \ref{Ocost}. Obviously the impact of the outages depends on the duration and the frequency of the events. We adopt the statistics conducted by \cite{EPRI} and assume that the consumer average interruption duration in one year is $88$ minutes. 
\item \emph{Reduced Time of Use (TOU) charges} include the savings occurred by eliminating the use of peak hour electricity, and using storage units. ``Time Of Use'' tariffs may vary in different territories and for different seasons. In Table \ref{TOURates}, we present the average TOU tariffs in the U.S.
\item \emph{Demand charges} is a significant portion of the commercial and industrial customers' bill. Some utilities also apply these charges to residential customers. It is usually computed by the amount and the duration of the peak usage \cite{oudalov2007sizing}. Since it is not easy to compute these charges without knowing the load profile, we did not consider this in our calculations. An overview of demand charges is presented in Table \ref{DCost}.

\end{enumerate}
It is noteworthy that our goal in discussing this example scenario is to show how our framework can make the ESS deployment economically viable. Using ESSs is mostly application-dependent and there can more benefits if storage is coupled with distributed generation options. Additional benefits will include profits by energy trading and reduced electricity bills. Also, the cost of carbon emission decreases (due to the elimination of peaking generators). The exact amount of savings depends on the specific policies and regulations. For commercial and industrial customers, there will be additional benefits due to demand charges. 
\begin{figure*}[t]
        \centering
                \begin{subfigure}[b]{0.32\textwidth}
                \centering
                                \includegraphics[width=\columnwidth]{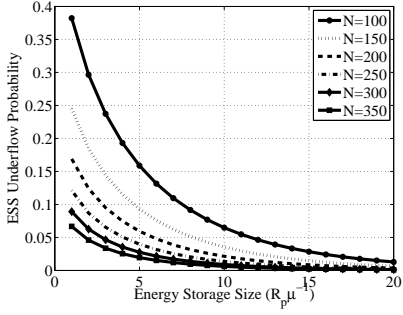}
 \caption{Community storage sizing for different user population.}\label{NvsESS}
                      \end{subfigure}
 \;
        \begin{subfigure}[b]{0.32\textwidth}
                \centering
                 \includegraphics[width=\columnwidth]{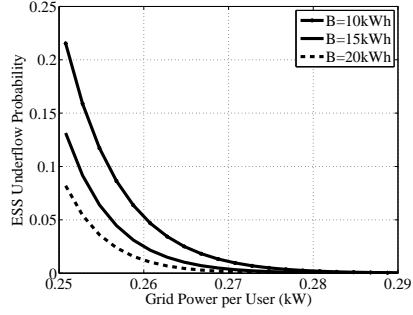}
 \caption{Community storage sizing for varying grid power, $N=200$ appliances (users). }\label{capacityOverFlow}
       \end{subfigure}%
        \;
        \begin{subfigure}[b]{0.32\textwidth}
                \centering
                \includegraphics[width=\columnwidth]{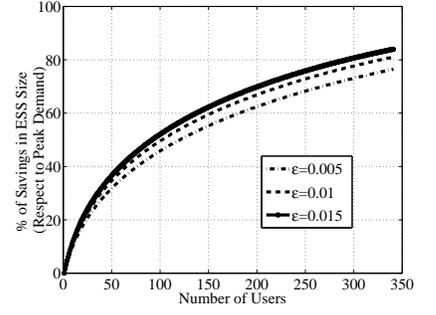}
 \caption{Savings in ESS Size for varying user population}\label{Savings1}
                     \end{subfigure}
        \caption{Evaluation of single-class case.}\label{econ}
\end{figure*}

\section{Numerical Examples}\label{examples}
\subsection{Single Class Customers}
In this section, we provide numerical examples to explain the system dynamics and show how the proposed framework can be used in typical peak shaving applications. In the first group of setting, uniform customer demand is considered ($K=1$). We use the aforementioned normalized values (unit time is measured in $\mu^{-1}$ and unit demand is measured in peak demand - $R_p$). We start by exploring the relationships between the number of users ($N$), ESS size (in $R_p \mu^{-1}$ units) and the corresponding underflow probability for a given system capacity, which remains constant over time, $C$.  Charge request rate per single user $\lambda$ is set to $0.3$, and the mean capacity above the mean demand per user is set to $\upsilon=0.035$ (assumptions are valid for all single class case studies). Therefore, the total system capacity becomes $C=0.2658N$ units. In Fig.~\ref{NvsESS}, sizing problem is evaluated for user population $N$ ranging from $100$ to $350$. Considering the fact that, one household can employ $10-12$ appliances, this interval is chosen to represent a multi-dwelling building or a typical building on a university campus. These findings can be used in the following ways. First, system operators can provision the storage units and provide statistical guarantees (underflow probability) to their customers for a given user population $N$. For example, for a large-scale electric vehicle charging lot (e.g., located in shopping mall or airports~\cite{sgc13}) with $N=350$ charging slots in order to accommodate $99.95\%$ of the customer demand the ESS size should be selected as $B=7\times R_{p}\times\mu^{-1}=7\times6.6\times 2 = 92.4$ kWh (considering level-II chargers and $30$ minutes as the unit time). It is noteworthy that the required storage size decreases as the user population increases. 

Next, we consider the case where the system operator employs an already acquired ESS of sizes $B=10$kWh, $15$kWh, or $20$kWh. In this case, the critical step is to calculate the minimum amount of power to draw from the grid so that performance guarantees can be achieved. This case is evaluated in in Fig.~\ref{capacityOverFlow} for $N=200$ appliances. For instance, suppose ESS size $B=10$kWh is already acquired and the goal is to meet $95$\% of the demand, and then system operator should draw $0.26\times200=52$kW from the grid.
The primary motivation for the employment of the ESS is to reduce the stress on the grid and improve the utilization of power system components in a cost effective manner. Hence, our last evaluation is on computing the amount of savings in ESS size with respect to current common sizing practice, allocating peak demand. The results are depicted in Fig. \ref{Savings1} show that instead of sizing the ESS to meet the entire customer demand, just by rejecting a few percentage of customers, considerable savings in the storage size can be achieved. 
\begin{figure*}[t]
        \centering
                \begin{subfigure}[b]{0.32\textwidth}
                \centering
                                \includegraphics[width=\columnwidth]{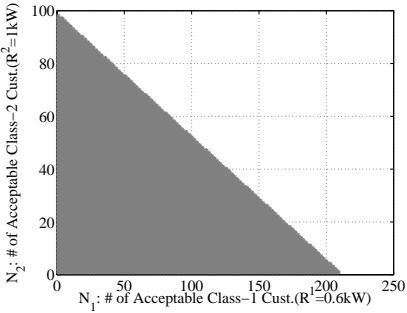}
 \caption{Admission Region for $K=2$.}\label{MRes1}
                      \end{subfigure}
 \;
        \begin{subfigure}[b]{0.32\textwidth}
                \centering
                 \includegraphics[width=\columnwidth]{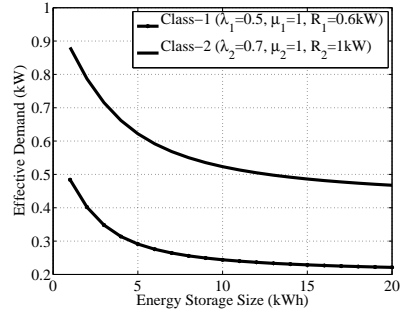}
 \caption{\emph{Effective Demand} Calculation for $K=2$.}\label{MRes2}        \end{subfigure}%
        \;
        \begin{subfigure}[b]{0.32\textwidth}
                \centering
                \includegraphics[width=\columnwidth]{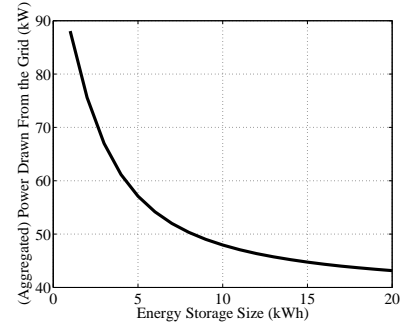}
 \caption{ESS Sizing for $K=2$.}\label{MRes3}
                     \end{subfigure}
        \caption{Evaluation of multi-class case.}\label{econ}
\end{figure*}

\subsection{Multi-Class Customers}
Next, we consider the case where customers can request different demand levels, hence different customer classes. In this subsection, our main objective is to compute the ${\omega _k}$ (\emph{effective demand}) parameter for each class such that ${\omega _k}$ replaces $R_{k}$ and the overflow probability targets are met. We start by discussing a toy example. Suppose that users can generate four different levels of demands, with the following parameters $\bm{\lambda}=\{0.3, 0.5, 0.7, 0.9\}$ , $\bm{\mu}=\{1, 1, 1, 1\}$, and $\bm{R}^{k}={0.2, 0.4, 0.6, 0.8}$ (in kW). Then for the target outage $\varepsilon=10^{-4}$,  and ESS size $B=10$kWh, the corresponding \emph{effective demand} becomes $\bm{\omega}=\{0.0515,\; 0.1567,\; 0.2958,\; 0.4550\}$. Alternatively, the system operator can find the minimum storage size with respect to the available grid resources and target performance metric. For the simplicity of presentation, the remaining numerical results consider two classes of customers with the following parameters $\lambda_1=0.5$, $\mu_1=1$, $R_{1}=0.5$kW and $\lambda_2=0.7$, $\mu_2=1$, $R_{2}=1$kW. Our first evaluation shows the admission set for $C=50$kW and $B=10$kWh, for a target overflow probability of $\varepsilon=0.001$. Results presented in Fig. \ref{MRes1} shows that station operator can choose from any set of customer numbers from the shaded region. Obviously, since the demand for Class $1$ customers is less than the other class, the system operator can accept more customers from Class $1$. As a second evaluation, we calculate the \emph{Effective Demand} for varying storage size. This time the target underflow probability is set as $\varepsilon=0.0005$ and the customer population is chosen as $N_{1}=100$ and $N_{2}=45$. The results depicted in Fig. \ref{MRes2} shows that instead of provisioning the system according to peak demand ($R_{1}=0.6kW$, $R_{2}=1$kW), the use of \emph{effective demands} (e.g., for $B=10$kWh, $\omega_{1}=0.25$ and $\omega_{2}=0.5$) reduces the provisioning of resources tremendously. Our final results are on the ESS sizing for multi classes. For the same set of parameters, we evaluate the storage size with respect to power drawn from the grid. Similar to the single class case, this result can be used to size the storage unit for a given grid power, or it can be used to compute the required grid resources for a given storage size.

\begin{figure*}[t]
        \centering
                \begin{subfigure}[b]{0.32\textwidth}
                \centering
                                \includegraphics[width=\columnwidth]{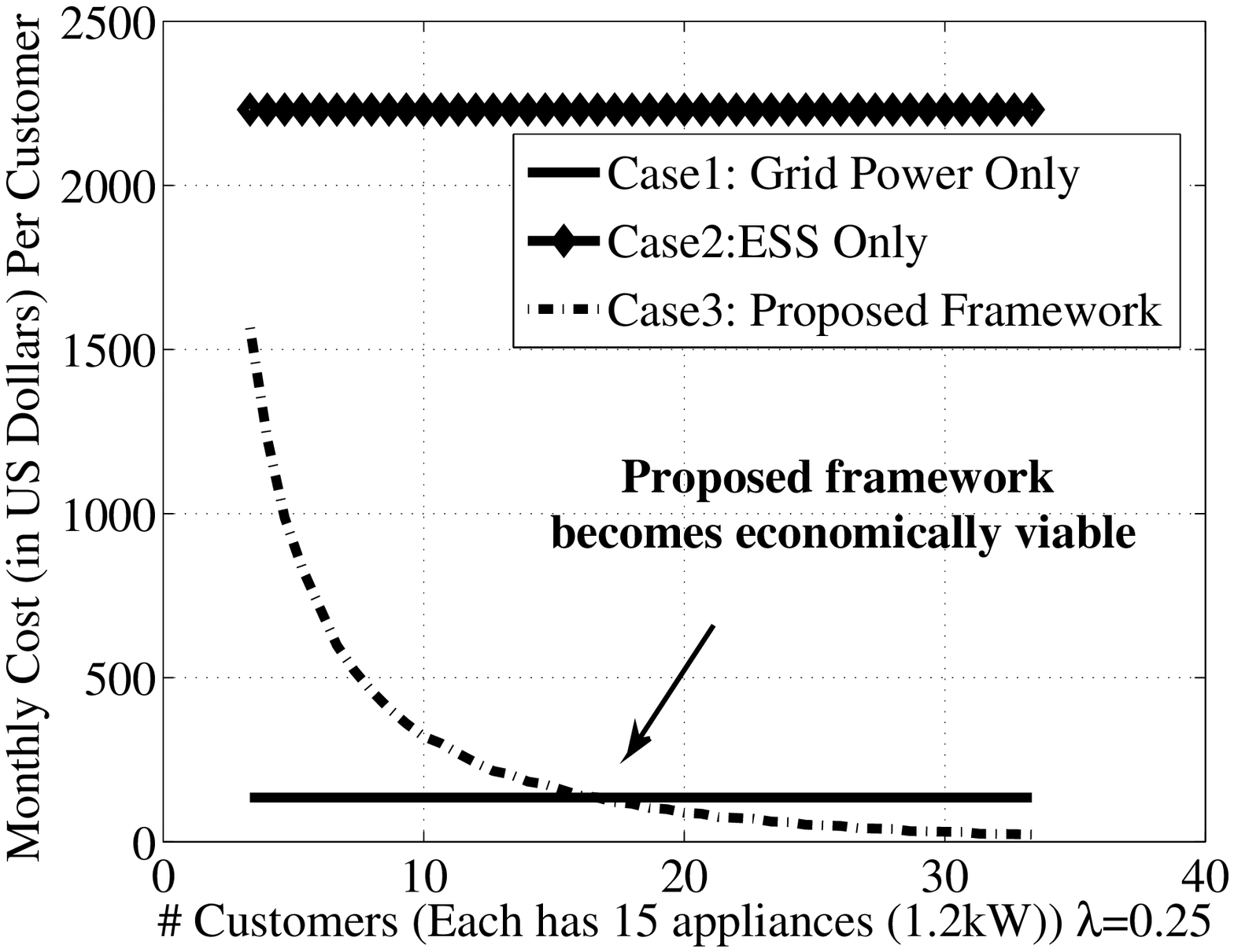}
 \caption{ESS becomes viable when more than $16$ customers share the cost for $\lambda=0.25$.}\label{econ1}
                     \end{subfigure}
 \;
        \begin{subfigure}[b]{0.32\textwidth}
                \centering
                 \includegraphics[width=\columnwidth]{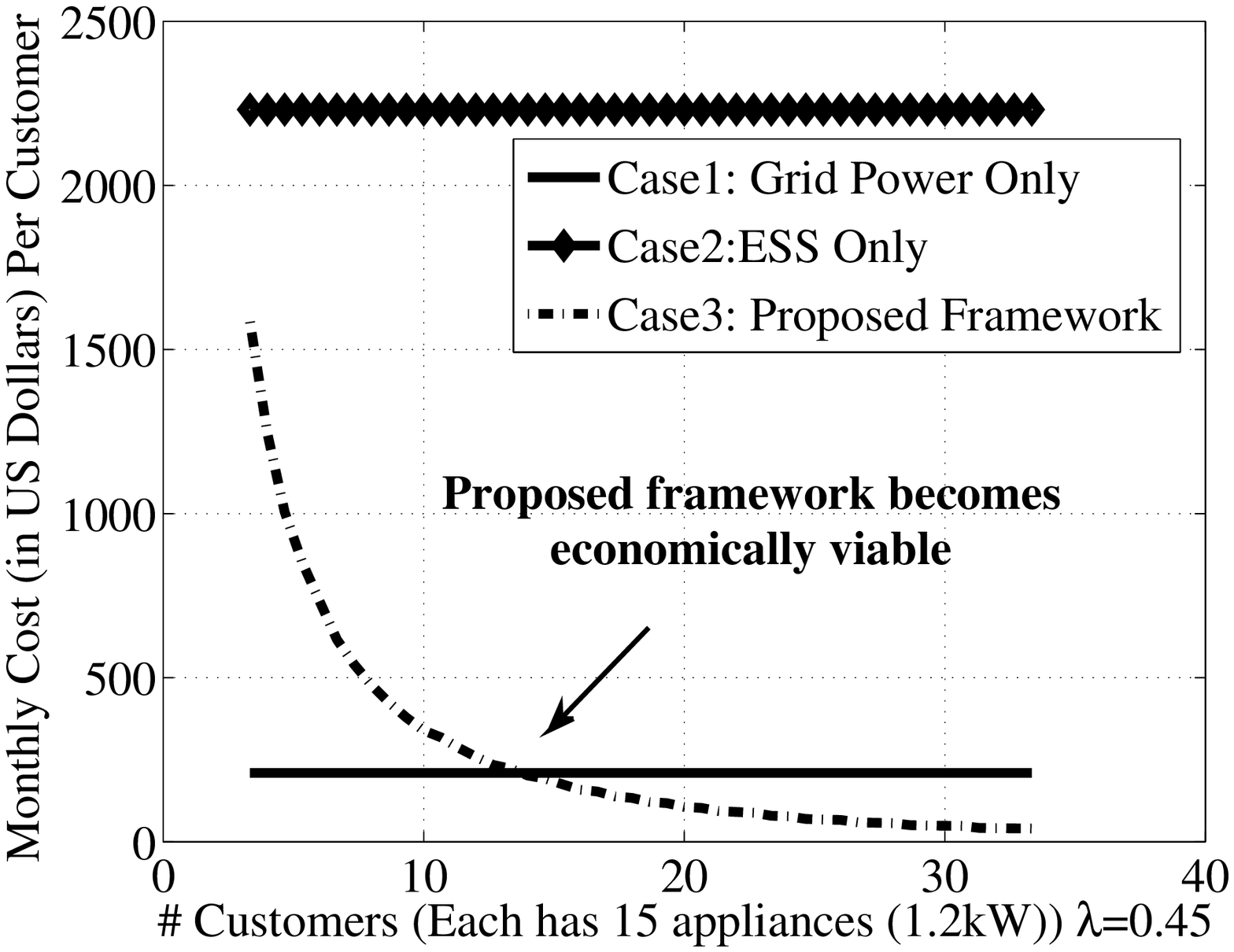}
 \subcaption{ESS becomes viable when more than $11$ customers share the cost for $\lambda=0.45$.} \label{econ2}
        \end{subfigure}%
        \;
        \begin{subfigure}[b]{0.32\textwidth}
                \centering
                \includegraphics[width=\columnwidth]{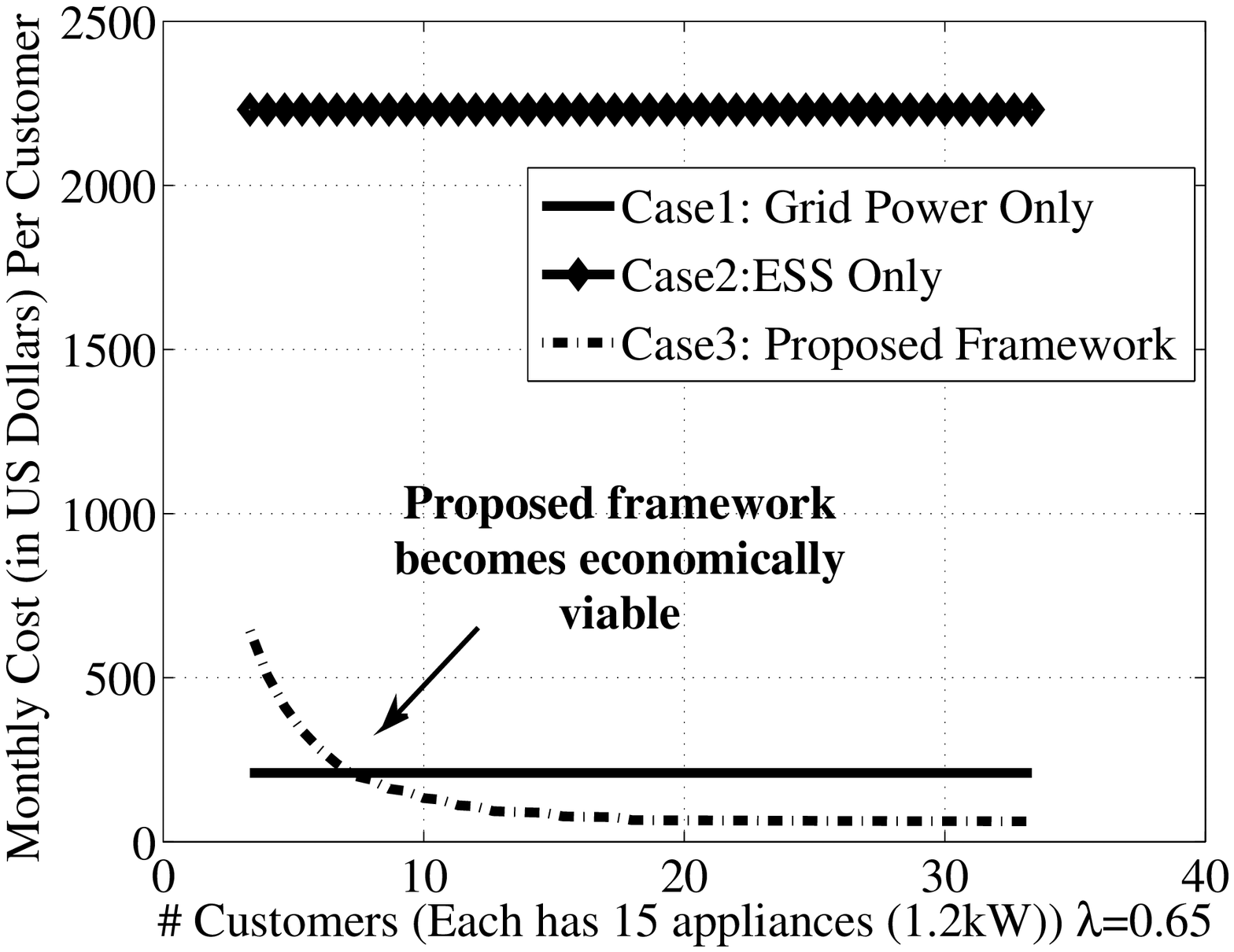}
 \subcaption{ESS becomes viable when more than $7$ customers share the cost for $\lambda=0.65$.}\label{econ3}
                     \end{subfigure}
        \caption{Economical evaluation and cost analysis of the proposed framework in US Dollars (\$) per user per one month.}\label{econ}
        \vspace{-8 pt}
\end{figure*}

\subsection{ESS Economic Analysis}
In this subsection, we provide several numerical examples with actual real-world scenarios to show how the proposed framework reduces the cost of the customers. In our evaluations, we choose the average cost per residential customer. We assume that peak appliance demand is set to $R_{1}=1.2$kW, each user employs $15$ appliances, and there are $250$ days of peak usage in one year.  For simplicity, the peak hour duration is assumed to 1 hour. All the cost/benefit calculations are normalized to peak hour demand. We consider three cases and compare the cost of system operation for different customer population and the frequency of appliance usage.
\begin{itemize}
\item \emph{Case 1 - Grid Power Only:} In this setting, all of the peak hour demand is met by the power demand. This case reflects the current state of affairs. The cost includes the cost of peak charging (\$$0.25$kWh) and all of the listed benefits except the demand charges. Obviously, since there is no storage unit, users suffer from power quality and reliability, and high peak-hour charges.
\item \emph{Case 2 - ESS Only:} In this case, we consider the case where all of the peak hour demand is met by the storage unit. The cost includes the cost of off-peak hour charging and the storage cost. In this case, all of the listed benefits except the demand charges contribute to reducing the total cost of the system.
\item \emph{Case 3 - Proposed Framework:} This is the case where we compute the storage size based on the mathematical proposed framework. We set the probability target as $\varepsilon=0.001$. By noting that, there can be more than one $(C,B)$ combinations that satisfy the overflow targets for the evaluations, we assume that grid power is $20\%$ more than the average customer demand.
\end{itemize}
According to Electric Power Research Institute \cite{EPRI}, the greatest cost of such of energy management and Time-Of-Use pricing applications stems from the acquisition and the Operation \& Management costs of the storage units\footnote{Another highly relevant report is presented in \cite{sandiaReport}, however we consider \cite{EPRI} as it is more recent.}. 
Generally, the lifetime of a storage unit is assumed to be $15$ years with $10$\% discount rate\footnote{This is used to compute the Net Present Value of the ESS. See \cite{EPRI} for further details.}. In this paper, we assume that the unit cost of the storage unit per year is \$$1500$, which is the typical cost for a Li-ion battery.

We compute the aggregate cost of each case per user per month for three different demand frequencies  $\lambda=0.25$, $\lambda=0.45$, and $\lambda=0.65$. The results depicted in Fig. \ref{econ} have the following interpretations. Since the cost of ESS technologies is high, the stand-alone ESS setting (case 2) is extremely costly and not useful practically, which is in line with the conclusion of EPRI \cite{EPRI} and SANDIA \cite{sandiaReport} studies. On the other hand, users can enjoy the aforementioned benefits by sharing the storage (e.g., in a neighborhood or on a campus etc.), as more users share the same resource pool, the cost per user decreases. Obviously, the population of users who share the same ESS depends on the physical deployments. For instance, $100$ apartments in highly dense regions can be fed by the same storage, while this may not be possible in a wide area due to losses. Further, as the storage usage increases (higher $\lambda$) the economic benefits improve faster. These evaluations show that, it may be practical to share the same storage unit so that the total cost of ownership reduces, and customers enjoy a variety of benefits. One final note is that the same analysis can be applied to the multi-class case by computing the \emph{effective demands} and the corresponding storage size. However, our main goal is to conduct the economic analysis per user based. 
 \section{Conclusion}
  In this paper, we have developed a stochastic analytical framework to provision a sharing-based energy storage units, which are expected to be employed in peak hour energy management systems for residential customers. The analysis establishes the interplay among dynamic grid capacity, the number of consumers, different appliances types, and the guarantee levels for avoiding outage events. We have provided a detailed economical analysis and have shown that ESS at residential level is economically beneficial if employed in a sharing-based architecture.

\section*{Acknowledgment}
This publication was made possible by NPRP grant \# 6-149-2-058 from the Qatar National Research Fund (a member of Qatar Foundation). The statements made herein are solely the responsibility of the authors.

\ifCLASSOPTIONcaptionsoff
  \newpage
\fi

\bibliographystyle{IEEEtran}
\bibliography{energycon}

\end{document}